\documentclass[a4paper, 10pt]{article}
\usepackage[latin1]{inputenc}
\usepackage[T1]{fontenc}      
\usepackage[english]{babel}  
\usepackage{amsmath}
\usepackage{amsthm}
\usepackage{amssymb}
\usepackage{amsfonts}
\usepackage{mathrsfs}
\usepackage{stmaryrd}
\usepackage{subeqnarray} 
\usepackage{lscape}
\usepackage{color}
\usepackage{cases}
\usepackage{dsfont}
\usepackage{lscape}
\usepackage{graphicx}
\usepackage{subfigure}
\newtheorem{Lemma}{Lemma}
\newtheorem{Proposition}{Proposition}
\newtheorem{Remark}{Remark}
\newtheorem{Theorem}{Theorem}

\newtheorem{Definition}{Definition}
\usepackage{lettrine}
\usepackage{geometry}
\usepackage{hyperref}
\usepackage{caption}
\usepackage{dsfont}
\hypersetup{
    colorlinks=true,                         
    linkcolor=red, 
    citecolor=blue, 
    urlcolor=green  } 

\usepackage{enumitem}
\newlist{exo}{enumerate}{1}
\setlist[exo]{label=Case \arabic* :}

\begin{document}

\title{Convergence for PDEs with an arbitrary odd order spatial derivative term}
\author{Cl\'ementine Court\`es$^1$}
\date{November 28, 2016}
\footnotetext[1]{Laboratoire de Math\'ematiques d'Orsay, Univ. Paris-Sud, CNRS, Universit\'e
Paris-Saclay, 91405 Orsay, France.
}

\maketitle

\abstract{We compute the rate of convergence of forward, backward and central finite difference $\theta$-schemes for linear PDEs with an arbitrary odd order spatial derivative term. We prove convergence of the first or second order for smooth and less smooth initial data.}
\section{Introduction}
\label{sec:1}
We study in this paper linear partial differential equations with an arbitrary odd order spatial derivative term, which read
\begin{eqnarray}
\partial_t u+\partial_x^{2p+1}u=0,
\label{EQ_BASE}
\end{eqnarray}
with $p\in \mathbb{N}$. The particular case $p=0$ corresponds to the advection equation with a unit constant speed $\partial_t u+\partial_x u=0$ and describes the passive advection of scalar field carried at constant speed. The case $p=1$ leads to Airy equation $\partial_tu+\partial_x^3 u=0$ that models the
propagation of long waves in shallow water \cite{Witham_1974} and derives from a linearization of the Korteweg-de Vries equation \cite{Craig_Goodman_1990}. We especially focus on the initial value problem where (\ref{EQ_BASE}) is considered with the initial condition $
u_{|_{t=0}}=u_0.
$
We deal with the numerical approach of this Cauchy problem and study the convergence of several finite difference schemes. Our concern here is to find a rate of convergence without assuming the smoothness of the initial data. \\
For this purpose, we use the finite difference method to discretize (\ref{EQ_BASE}) in $\mathbb{R}\times[0,T]$. We choose to deal with a uniform time and space discretization. Let $\Delta t>0$ and $\Delta x>0$ be the time and space steps, we note $t^n=n\Delta t$ for all $n\in\{0,...,N\}$ where $N=\lfloor\frac{T}{\Delta t}\rfloor$ and $x_j=j\Delta x$ for all $j\in\mathbb{Z}$. We denote by $\left(v_j^n\right)_{(j,n)}$ the discrete unknowns defined by
\begin{eqnarray}
\label{EQ_NUMERIQUE}
\left\{\begin{split}
&v_{j}^{n+1}+\theta\Delta t\left(D_{\bullet}^{2p+1}v\right)_j^{n+1}=v^n_j-(1-\theta)\Delta t\left(D_{\bullet}^{2p+1}v\right)_j^{n}, \forall (j,n)\in\mathbb{Z}\times\{0,.., N\},\\
&v_j^0=\frac{1}{\Delta x}\int_{x_j}^{x_{j+1}}u_0(y)dy, \forall j\in\mathbb{Z},
\end{split}\right.
\end{eqnarray}
with
\begin{subeqnarray}
\label{DEF_SCHEMAS}
\slabel{DEF_SCHEMAS_1}
&&\left(D_{\bullet}^{2p+1}v\right)_j^n=\left(D_+^{2p+1}v\right)_j^n=\sum_{k=0}^{2p+1}\frac{\binom{2p+1}{k}(-1)^k}{\Delta x^{2p+1}}v^n_{p-k+j+1} \textrm{\ (forward\ scheme)},\\
\slabel{DEF_SCHEMAS_2}
\hspace*{-0.4cm}\mathrm{or\ }&&\left(D_{\bullet}^{2p+1}v\right)_j^n=\left(D_-^{2p+1}v\right)_j^n=\sum_{k=0}^{2p+1}\frac{\binom{2p+1}{k}(-1)^k}{\Delta x^{2p+1}}v^n_{p-k+j} \textrm{\ (backward\ scheme)},\\
\slabel{DEF_SCHEMAS_3}
\hspace*{-0.4cm}\mathrm{or\ }&&\left(D_{\bullet}^{2p+1}v\right)_j^n=\left(D_c^{2p+1}v\right)_j^n=\frac{1}{2}\left(D_+^{2p+1}v+D_-^{2p+1}v\right)_j^n\textrm{\ (central\ scheme)}.
\end{subeqnarray}
The parameter $\theta$ belongs to $[0,1]$ and we recover the explicit scheme for $\theta=0$ and the implicit scheme for $\theta=1$.\\
\textit{Notations 1.}
We denote by $\mathbb{H}^s(\mathbb{R})$ (with $s>0$) the Sobolev space defined with the norm
$
\left| \left|u\right|\right|_{\mathbb{H}^s(\mathbb{R})}=\left(\int_{\mathbb{R}}\left(1+|\xi|^2\right)^s\left| \widehat{u}\left(\xi\right)\right|^2d\xi\right)^{\frac{1}{2}},
$
where $\widehat{u}$ is the Fourier transform of $u$. Moreover,
we use the standard $\ell^{\infty}\left(0,N; \ell^{2}_{\Delta}\left(\mathbb{Z}\right)\right)$ space whose norm is
$
\left|\left| v\right| \right|_{\ell^{\infty}\left(0,N; \ell^{2}_{\Delta}\left(\mathbb{Z}\right)\right)}=\underset{n\in\{0,..,N\}}{\mathrm{sup}}\sqrt{\sum_{j\in\mathbb{Z}}\Delta x|v_j^n|^2}.
$
Lastly, we note 
$
A\lesssim B
$
when  
$A\leq C B$ where $C$ is a constant independent of $\Delta x$ and $\Delta t$.
\section{Order of accuracy for an initial datum in $\mathbb{H}^{4p+2}(\mathbb{R})$}
\label{_Order_of_accuracy_for_a_smooth_initial_data_}
We hereafter find some condition on $\theta$, $\Delta t$ and $\Delta x$ for the schemes to be consistent and stable, to conclude the convergence study according to the Lax-Richtmyer theorem \cite{Lax_Richtmyer_1956}.
\subsection{Consistency estimate}
\label{_Consistency_estimate_}
In Sect.~\ref{_Order_of_accuracy_for_a_smooth_initial_data_}, we suppose the initial datum regular enough to compute all the needed derivatives and the Taylor expansions up to the desired order. Indeed, supposing $u_0$ regular is sufficient to ensure the same regularity for $u(t,.)$ for all $t\in[0,T]$ because of the following result.
\begin{Remark}
\label{REM_1}
Let $u$ be a solution of (\ref{EQ_BASE}), then by linearity of the equation all the derivatives of $u$ verify (\ref{EQ_BASE}) too and by Fourier transform, the $\mathbb{L}^2$--norm of all its derivatives are conserved :
$
||\partial_x^ku(t,.)||_{\mathbb{L}^2(\mathbb{R})}=||\partial_x^ku_0||_{\mathbb{L}^2(\mathbb{R})},
$
for all $k\in\mathbb{N}$. Thus, $u_0\in\mathbb{H}^{4p+2}(\mathbb{R})$ implies $u(t,.)\in\mathbb{H}^{4p+2}(\mathbb{R}) ,\ \forall t\in[0,T]$.
\end{Remark}
\begin{Definition}
\label{DEF_EPSILON} For all $(j,n)\in\mathbb{Z}\times\{0,..,N\}$, we note $\left(u_{\Delta}\right)_j^n=\frac{1}{\Delta x}\int_{x_j}^{x_{j+1}}u(t^n,y)dy$ with $u$ the exact solution of the Cauchy-problem \eqref{EQ_BASE} from $u_0$. For  all $ (j,n)\in\mathbb{Z}\times\{0,.., N\}$, the consistency error is defined as
\begin{eqnarray*}
\epsilon_j^n=\frac{\left(u_{\Delta}\right)_{j}^{n+1}-\left(u_{\Delta}\right)^n_j}{\Delta t}+\theta\left(D_{\bullet}^{2p+1}u_{\Delta}\right)_j^{n+1}+(1-\theta)\left(D_{\bullet}^{2p+1}u_{\Delta}\right)_j^{n},
\end{eqnarray*}
with $D_{\bullet}^{2p+1}$ defined by (\ref{DEF_SCHEMAS_1})--(\ref{DEF_SCHEMAS_3}).
\end{Definition}
\begin{Proposition}
Assume $u_0\in\mathbb{H}^{4p+2}(\mathbb{R})$ (and $u_0\in\mathbb{H}^{6p+3}(\mathbb{R})$ if $\theta=\frac{1}{2}$) then, for the forward or backward finite difference schemes (\ref{DEF_SCHEMAS_1}) and (\ref{DEF_SCHEMAS_2}), the following consistency inequality holds
\begin{small}
\begin{eqnarray*}
\left|\left|\epsilon\right|\right|_{\ell^{\infty}\left(0,N; \ell^{2}_{\Delta}\left(\mathbb{Z}\right)\right)} \lesssim \Delta t\left|\frac{1}{2}-\theta\right|||\partial_x^{4p+2}u_0||_{\mathbb{L}^2(\mathbb{R})}+\Delta x||\partial_x^{2p+2}u_0||_{\mathbb{L}^2(\mathbb{R})}+\Delta t^2\left|\left|\partial_x^{6p+3}u_0\right|\right|_{\mathbb{L}^2(\mathbb{R})}.
\end{eqnarray*}
\end{small}
For the central finite difference scheme (\ref{DEF_SCHEMAS_3}), the consistency inequality is as follows
\begin{small}
\begin{eqnarray*}
\left|\left|\epsilon\right|\right|_{\ell^{\infty}\left(0,N; \ell^{2}_{\Delta}\left(\mathbb{Z}\right)\right)} \lesssim \Delta t\left|\frac{1}{2}-\theta\right|||\partial_x^{4p+2}u_0||_{\mathbb{L}^2(\mathbb{R})}+\Delta x^2||\partial_x^{2p+3}u_0||_{\mathbb{L}^2(\mathbb{R})}+\Delta t^2\left|\left|\partial_x^{6p+3}u_0\right|\right|_{\mathbb{L}^2(\mathbb{R})}.
\end{eqnarray*}
\end{small}
\label{PROP_1}
\end{Proposition}
Before proving the previous result, we state a useful lemma.
\begin{Lemma}
For all $\ell$ and $p$ in $\mathbb{N}$, there exists $\xi\in]-p, p+1[$ such that
\begin{small}
\begin{eqnarray*}
\sum_{k=0}^{2p+1}\binom{2p+1}{k}(-1)^k(p-k+1)^{\ell}=\left\{\begin{split}
&0 &\mathrm{\ if\ }\ell<2p+1,\\
&\ell ! &\mathrm{\ if\ }\ell=2p+1,\\
&\frac{\ell !}{(\ell-2p-1)!}\xi^{\ell-2p-1} &\mathrm{\ if\ }\ell>2p+1.
\end{split}
\right.
\label{eq_gl}
\end{eqnarray*}
\end{small}
\label{LEMMA_1}
\end{Lemma}
\begin{proof}[Proof of Lemma~\ref{LEMMA_1}]
 
Let $(x_{j-p}, ..., x_{j+p+1})$ be $2p+2$ points regularly spaced of $h$, we recall the divided difference of order $2p+2$ of a smooth function $f$ :

\begin{equation}
(2p+1)!f[x_{j-p}, ..., x_{j+p+1}]=\frac{\sum_{k=0}^{2p+1}\binom{2p+1}{k}(-1)^kf(x_{p-k+j+1})}{h^{2p+1}}.
\label{EQ_diff_Divisees}
\end{equation}
Moreover, we recall the existence of $\xi\in]\mathrm{min}(x_{j-p}, ..., x_{j+p+1}), \mathrm{max}(x_{j-p}, ..., x_{j+p+1})[$ such as 
\begin{equation*}
(2p+1)!f[x_{j-p}, ..., x_{j+p+1}]=f^{(2p+1)}(\xi).
\end{equation*}
For more details, please refer to \cite{Demailly_Livre}. Lemma~\ref{LEMMA_1} is a consequence of the two previous equations with $f:y\mapsto y^{\ell}$, $h=1$, $j=0$ and $x_i=i$ for $i\in\mathbb{Z}$.
 
\end{proof}

\begin{proof}[Proof of Proposition~\ref{PROP_1}]
 
For $u_0\in\mathbb{H}^{4p+2}\left(\mathbb{R}\right)$ and for the forward finite difference scheme, one has
\begin{small}
\begin{eqnarray*}
\left(D_{\bullet}^{2p+1}u_{\Delta}\right)_j^{n}=\left(D_{+}^{2p+1}u_{\Delta}\right)_j^{n}=\sum_{k=0}^{2p+1}\frac{\binom{2p+1}{k}(-1)^k}{\Delta x^{2p+1}}\left(\frac{1}{\Delta x}\int_{x_j}^{x_{j+1}}u(t^{n},y+(p-k+1)\Delta x)dy \right).
\end{eqnarray*}
\end{small}
Using a Taylor expansion (in space) up to order $2p+2$ and exchanging the two sums inside leads to
\begin{small}
\begin{eqnarray}
\left(D_{+}^{2p+1}u_{\Delta}\right)_j^{n}=\frac{1}{\Delta x}\int_{x_j}^{x_{j+1}}\sum_{\ell=0}^{2p+1}\frac{\partial_x^{\ell} u(t^n,y)}{\ell !}\frac{\left(\sum_{k=0}^{2p+1}\binom{2p+1}{k}(-1)^k(p-k+1)^{\ell}\Delta x^{\ell}\right)}{\Delta x^{2p+1}}dy+\left(R_+\right)_j^n,
\label{EQ_avec_R}
\end{eqnarray}
\end{small}
where
\begin{small}$$\left(R_+\right)_j^n=\frac{1}{\Delta x}\int_{x_j}^{x_{j+1}}\int_{y}^{y+(p-k+1)\Delta x}\frac{\partial_x^{2p+2}u(t^n,z)}{(2p+1)!}\left(\frac{\sum_{k=0}^{2p+1}\binom{2p+1}{k}(-1)^k(y+(p-k+1)\Delta x-z)^{2p+1}}{\Delta x^{2p+1}}\right)dzdy.$$\end{small}
For simplicity, we will only use $||R_+^n||_{\ell^2_{\Delta}}\lesssim\Delta x||\partial_x^{2p+2}u(t^n,.)||_{\mathbb{L}^2(\mathbb{R})}. $
Equation~(\ref{EQ_avec_R}) is simplified thanks to Lemma~\ref{LEMMA_1}.
Eventually, we obtain
\begin{small}
$
\left(D_{+}^{2p+1}u_{\Delta}\right)_j^{n}=\frac{1}{\Delta x}\int_{x_j}^{x_{j+1}}\partial_x^{2p+1}u(t^n,y)dy+\left(R_+\right)_j^n.
$
\end{small}
Similarly, by adapting the previous computation, one has
\begin{small}
\begin{subeqnarray*}
\left(D_{+}^{2p+1}u_{\Delta}\right)_j^{n+1}=\frac{1}{\Delta x}\left(\int_{x_j}^{x_{j+1}}\partial_x^{2p+1}\right.&&u(t^n,y)dy+\int_{x_j}^{x_{j+1}}\Delta t\partial_t\partial_x^{2p+1}u(t^n,y)dy\\
&&\left.+\int_{x_j}^{x_{j+1}}\int_{t^n}^{t^{n+1}}\partial_t^2\partial_x^{2p+1}u(s,y)(t^{n+1}-s)dsdy\right)+\left(R_+\right)_j^{n+1}.
\end{subeqnarray*}
\end{small}
In order to compute the difference $\frac{\left(u_{\Delta}\right)_{j}^{n+1}-\left(u_{\Delta}\right)^n_j}{\Delta t}$ that appears in Definition~\ref{DEF_EPSILON}, we perform a Taylor expansion (in time) up to order 3.
 Gathering all those results together yields 
 \begin{small}
 \begin{subeqnarray*}
\epsilon_j^n &&=\frac{\Delta t}{2\Delta x}\int_{x_j}^{x_{j+1}}\partial_t^2 u(t^n,x)dx+\frac{1}{\Delta t\Delta x}\int_{x_j}^{x_{j+1}}\int_{t^n}^{t^{n+1}}\partial_t^3u(s,y)\frac{(t^{n+1}-s)^2}{2}dsdy\\&&+\frac{\theta \Delta t}{\Delta x}\int_{x_j}^{x_{j+1}}\partial_t\partial_x^{2p+1}u(t^n,y)dy+\frac{\theta}{\Delta x}\int_{x_j}^{x_{j+1}}\int_{t^n}^{t^{n+1}}\partial_t^2\partial_x^{2p+1}u(s,y)(t^{n+1}-s)dsdy\\&&+\theta \left(R_+\right)_j^{n+1}+(1-\theta)\left(R_+\right)_j^n.
 \end{subeqnarray*}
 \end{small}
The conclusion comes from the relation $\partial_tu(t,x) = -\partial_x^{2p+1}u(t,x)$, the Cauchy-Schwarz inequality and the conservation of the $\mathbb{L}^2$-norm (cf. Remark~\ref{REM_1}).

\end{proof}
\begin{Remark}
The regularity $\mathbb{H}^{4p+2}(\mathbb{R})$ (or $\mathbb{H}^{6p+3}(\mathbb{R})$ if $\theta=\frac{1}{2}$) comes from the Taylor expansion in time and is essential in this proof.
\label{REM_2_2_22}
\end{Remark}
\begin{Remark} We follow exactly the same guidelines for the backward finite difference scheme. For the central finite difference scheme, we need to perform a Taylor expansion in space up to order $2p+3$ to obtain \begin{small}
$
\left(D_{c}^{2p+1}u_{\Delta}\right)_j^{n}=\frac{1}{\Delta x}\int_{x_j}^{x_{j+1}}\partial_x^{2p+1}u(t^n,y)dy+\left(R_c\right)_j^n,
$
\end{small}
with 
$
||R_c^n||_{\ell^2_\Delta}\lesssim \Delta x^2||\partial_x^{2p+3}u_0||_{\mathbb{L}^2(\mathbb{R})}.
$
\end{Remark}
\subsection{Stability}
\label{_Stability_}
We note, for all $\left(v_j^n\right)_{j\in\mathbb{Z}}$ and $\xi\in[0,1]$, $\widehat{V^n}(\xi)=\sum_{k\in\mathbb{Z}}v_k^ne^{2i \pi k\xi}$ in $\mathbb{L}^2([0,1])$ with the equivalence of the norms : $\sum_{j\in\mathbb{Z}}\Delta x|v_j^n|^2 = \Delta x\int_0^1\left|\widehat{V^n}\left(\xi\right)\right|^2d\xi.$ Eventually, we define the shift operator $\mathcal{S}^{\ell}$ by $\mathcal{S}^{\ell}v^n=(v_{j+\ell}^n)_{j\in\mathbb{Z}}$ thus, $\widehat{\mathcal{S}^{\ell}V^n}=e^{-2i \pi \ell \xi}\widehat{V^n}$.  
\begin{Definition}
A scheme is said to be stable in $\ell^2_{\Delta}(\mathbb{Z})$, if there exists a constant $C$ independent of $\Delta t$ and $\Delta x$ such that, for $\left(v_j^n\right)_{(j,n)}$ verifying Equation~(\ref{EQ_NUMERIQUE}),
\begin{equation*}
\left| \left|v^{n+1}\right|\right|_{\ell^2_{\Delta}(\mathbb{Z})}\leq \left(1+C \Delta t\right)\left| \left|v^{n}\right|\right|_{\ell^2_{\Delta}(\mathbb{Z})}, \ \forall n\in\{0,...,N\}.
\end{equation*}
\end{Definition}
\begin{Proposition}
For small $\Delta t$ and $\Delta x$, the stability under the Courant-Friedrich-Lewy condition (in short CFL cond.) is explained in Table~\ref{TABLE_STABLE-TTable}.
\begin{table}[h]
\begin{center}
\begin{tabular}{p{1.5cm}p{3.5cm}p{3.5cm}p{3.5cm}}
\hline\noalign{\smallskip}
& \multicolumn{2}{c}{\hspace*{-1cm}$\overbrace{\hspace*{6cm}}^{p \mathrm{\ even}}$} &$p$ odd\\
& $p=0$ & $p\neq 0$&  \\
& (Advection) & & \\
\hline\noalign{\smallskip}
&&&\\
Forward & stable under the CFL cond.& unconditionally unstable&stable under the CFL cond.\\
scheme &$\Delta t(1-2\theta)\leq -\Delta x$&& $\Delta t(1-2\theta)\leq \frac{\Delta x^{2p+1}}{2^{2p}}$\\
&&&\\
Backward & stable under the CFL cond.& stable under the CFL cond.&unconditionally unstable\\
scheme &$\Delta t(1-2\theta)\leq \Delta x$&$\Delta t(1-2\theta)\leq \frac{\Delta x^{2p+1}}{2^{2p}}$&\\
&&&\\
Central & stable under the CFL cond.& stable under the CFL cond.&stable under the CFL cond.\\
scheme &$\Delta t(1-2\theta)\leq 2C \Delta x^2$&$\Delta t(1-2\theta)\leq 2C\frac{\Delta x^{4p+2}}{2^{4p}}$&$\Delta t(1-2\theta)\leq 2C\frac{\Delta x^{4p+2}}{2^{4p}}$\\
\noalign{\smallskip}\hline\noalign{\smallskip}
\end{tabular}
\caption{Stability results for finite difference $\theta$-schemes}
\label{TABLE_STABLE-TTable}       
\end{center}
\end{table}
\label{PROP_STABILITY}
\end{Proposition}
The following computation will simplify the proof of Proposition~\ref{PROP_STABILITY}.
\begin{Lemma}
One has, for all $\xi \in [0,1],$ \begin{center}
$
\sum_{k=0}^{2p+1}\binom{2p+1}{k}(-1)^ke^{-2i \pi(p-k+1)\xi} = e^{-i \pi \xi}\left( -2i \mathrm{sin}(\pi\xi)\right)^{2p+1}.
$
\end{center}
\label{LEMMA_2}
\end{Lemma}
\begin{proof}[Proof of Lemma~\ref{LEMMA_2}]
 
A proof may be found in Lemma 1.1 of \cite{Memoire_M2}.
 
 \end{proof}
\begin{proof}[Proof of Proposition~\ref{PROP_STABILITY}] 
The forward finite difference scheme~(\ref{DEF_SCHEMAS_1}) leads to
\begin{small}
\begin{eqnarray*}
&&\widehat{U^{n+1}}(\xi)\left( 1+\frac{\theta \Delta t}{(\Delta x)^{2p+1}}\sum_{k=0}^{2p+1}\binom{2p+1}{k}(-1)^ke^{-2i \pi(p-k+1)\xi} \right)\\
&&=\widehat{U^n}(\xi)\left( 1-\frac{(1-\theta)\Delta t}{(\Delta x)^{2p+1}}\sum_{k=0}^{2p+1}\binom{2p+1}{k}(-1)^k e^{-2i \pi(p-k+1)\xi}\right), \mathrm{\ for\ any\ } \xi \mathrm{\ in\ } [0,1].
\end{eqnarray*}
\end{small}
The two sums are simplified thanks to Lemma~\ref{LEMMA_2}.
We finally obtain
$
\widehat{U^{n+1}}(\xi)=A_{+}(\xi)\widehat{U^n}(\xi), 
$
with $A_{+}$ the amplification coefficient defined by, $\forall \xi\in[0,1]$
\begin{eqnarray}
A_{+}(\xi)=\frac{\left( 1-\frac{(1-\theta) \Delta t}{(\Delta x)^{2p+1}}e^{-i\pi \xi}(-i)^{2p+1} \left(2\mathrm{sin}(\pi\xi) \right)^{2p+1}\right)}{\left( 1+\frac{\theta \Delta t}{(\Delta x)^{2p+1}}e^{-i\pi \xi}(-i)^{2p+1} \left(2\mathrm{sin}(\pi\xi) \right)^{2p+1}\right)}.
\label{eq_entouree_decentre_droit}
\end{eqnarray}
We are looking for a condition ensuring $
\left| A_{+}(\xi)\right|^2<(1+C\Delta t)^2
$ for any $\xi$ in $[0,1]$.

\begin{exo}
\item Assume that the parameter $p$ of the spatial derivative is even. 
If $p\neq0$, the stability condition leads to
$
\frac{2^{2p}\Delta t}{(\Delta x)^{2p+1}}(\mathrm{sin}(\pi\xi))^{2p}(1-2\theta)\leq-1,
$ (cf. \cite{Memoire_M2})
which is impossible for all $\xi\in[0,1]$ : thus the forward finite difference scheme is unconditionally unstable for $p$ even and non zero.
On the contrary, assuming $p=0$ means that the forward finite difference scheme is stable under CFL condition : $\Delta t(1-2\theta)\leq -\Delta x$ (which implies $\theta>\frac{1}{2}$).

\item In this case, the parameter $p$ of the spatial derivative is odd, then
the sufficient condition becomes $
\frac{\Delta t}{(\Delta x)^{2p+1}}(2\mathrm{sin}(\pi\xi))^{2p}(1-2\theta)\leq1
$ (cf. \cite{Memoire_M2}).
 Then the forward finite difference scheme is stable under the CFL condition 
$
\Delta t(1-2\theta)\leq\frac{\Delta x^{2p+1}}{2^{2p}}.
$
Table~\ref{TABLE_STABLE-TTable} is a straightforward consequence.
\end{exo}

\end{proof}
\begin{Remark} For the backward finite difference scheme, the only difference in the amplification coefficient is $e^{i \pi\xi}$ instead of $e^{-i \pi\xi}$ (in both the numerator and denominator). The parity needed for the stability changes because of that difference.
For the central finite difference scheme, $e^{-i \pi\xi}$ is replaced with $\cos(\pi\xi)$ in the numerator and the denominator of the amplifiaction coefficient.
\end{Remark}
\subsection{Error estimates}
\label{_Error_estimates_}
We define the convergence error as follows.
\begin{Definition}
For all $j\in \mathbb{Z}$ and $n\in\{0,...,N\}$, for $u$ the analytical solution of (\ref{EQ_BASE}) from $u_0$ and $(v_j^n)_{(j,n)}$ the numerical solution of (\ref{EQ_NUMERIQUE}), the convergence error is denoted by $e_j^n$ and defined by
$
e_j^n=\frac{1}{\Delta x}\int_{x_j}^{x_{j+1}}u(t^n,y)dy-v_j^n.
$
\end{Definition}
We are now able to state the main result of this section.
\begin{Theorem} For an initial datum $u_0\in\mathbb{H}^{4p+2}(\mathbb{R})$ (and $u_0\in\mathbb{H}^{6p+3}(\mathbb{R})$ if $\theta=\frac{1}{2}$), the error estimate of the forward finite difference scheme (\ref{DEF_SCHEMAS_1}) (if $p$ is odd) or of the backward finite difference scheme (\ref{DEF_SCHEMAS_2}) (if $p$ is even) satisfies
\begin{small}
\begin{eqnarray*}
\left|\left|e\right|\right|_{\ell^{\infty}\left(0,N;\ell^2_{\Delta}\left(\mathbb{Z}\right)\right)}\lesssim\Delta t\left|\frac{1}{2}-\theta\right|\left|\left|\partial_x^{4p+2}u_0\right|\right|_{\mathbb{L}^2(\mathbb{R})}+\Delta x\left|\left|\partial_x^{2p+2}u_0\right|\right|_{\mathbb{L}^2(\mathbb{R})}+\Delta t^2\left|\left|\partial_x^{6p+3}u_0\right|\right|_{\mathbb{L}^2(\mathbb{R})}.
\end{eqnarray*}
\end{small}
For the central finite difference scheme (\ref{DEF_SCHEMAS_3}), the convergence rate becomes
\begin{small}
\begin{eqnarray*}
\left|\left|e\right|\right|_{\ell^{\infty}\left(0,N;\ell^2_{\Delta}\left(\mathbb{Z}\right)\right)}\lesssim\Delta t\left|\frac{1}{2}-\theta\right|\left|\left|\partial_x^{4p+2}u_0\right|\right|_{\mathbb{L}^2(\mathbb{R})}+\Delta x^2\left|\left|\partial_x^{2p+3}u_0\right|\right|_{\mathbb{L}^2(\mathbb{R})}+\Delta t^2\left|\left|\partial_x^{6p+3}u_0\right|\right|_{\mathbb{L}^2(\mathbb{R})}.
\end{eqnarray*}
\end{small}
All those results are gathered in Table~\ref{TABLE_ERROR_ESTIMATES}.
\end{Theorem}
\begin{proof}
We suppose $p$ odd, so we work with the forward finite difference scheme. The case $p$ even, with the backward scheme is similar. The definition of the convergence error implies
\begin{eqnarray*}
\widehat{e^{n+1}}(\xi)=A_{+}(\xi)\widehat{e^n}(\xi)+\frac{\Delta t}{1+\frac{\theta\Delta t}{\Delta x^{2p+1}}e^{-i \pi\xi}(-i)^{2p+1}(2\mathrm{sin}(\pi\xi))^{2p+1}}\widehat{\epsilon^n}(\xi).
\end{eqnarray*}
One has $\left|\left|\frac{1}{1+\frac{\theta\Delta t}{\Delta x^{2p+1}}e^{-i \pi\xi}(-i)^{2p+1}(2\mathrm{sin}(\pi\xi))^{2p+1}}\right|\right|_{\mathbb{L}^{\infty}([0,1])}\leq1$ and the stability condition gives $\left|\left|A_{+}\right|\right|_{\mathbb{L}^{\infty}([0,1])}\leq1+C\Delta t$. 
Thus, we obtain the following estimate, by discrete Gr\"onwall lemma
\begin{eqnarray*}
||e^{n+1}||_{\ell^2_{\Delta}}\leq (1+C\Delta t)||e^n||_{\ell^2_{\Delta}}+\Delta t||\epsilon^n||_{\ell^2_{\Delta}}\leq ...\leq e^{CT} ||e^0||_{\ell^2_{\Delta}}+\Delta te^{CT}\sum_{k=0}^{n}||\epsilon^k||_{\ell^2_{\Delta}}.
\end{eqnarray*}
The initial condition $v_j^0$, Eq.~(\ref{EQ_NUMERIQUE}), together with the consistency error conclude the proof.

\end{proof}
\begin{Remark}
As expected, for the particular case $\theta=\frac{1}{2}$ (the so-called Crank-Nicolson case), the rate of convergence in time is better as illustrated in Table~\ref{TABLE_ERROR_ESTIMATES}, provided $u_0\in\mathbb{H}^{6p+3}(\mathbb{R})$ (and not only in $\mathbb{H}^{4p+2}(\mathbb{R})$).
\end{Remark}
\section{Less smooth initial data}
\label{_For_any_initial_data_}
The previous order of accuracy is obtained for initial data $u_0$ at least in $\mathbb{H}^{4p+2}(\mathbb{R})$ (or $\mathbb{H}^{6p+3}(\mathbb{R})$ if $\theta=\frac{1}{2}$). In this section, our aim is to relax this hypothesis to obtain rates of convergence for non-smooth initial data, for example, $u_0\in \mathbb{H}^m(\mathbb{R})$ with $m>0$. We detail only the case $\theta\neq\frac{1}{2}$ but state the Crank-Nicolson results in Table~\ref{TABLE_ERROR_ESTIMATES}.

\subsection{Initial datum in $\mathbb{H}^{m}(\mathbb{R})$ with $m\geq 2p+2$}
\label{SUB_1_1}
As explained previously (Remark~\ref{REM_2_2_22}), the regularity of $u_0$ is determined by the Taylor expansion in time. A first step is then to deal with the time term in error estimates. The following proposition provides that the time error prevails until $u_0\in\mathbb{H}^{2p+2}(\mathbb{R})$, for which the spatial error becomes predominant.
\begin{Proposition}
\label{Prop_model}
Assume $u_0\in \mathbb{H}^{m}(\mathbb{R})$ with $m\geq 2p+2$, and let us fix $M=\mathrm{min}(m,4p+2)$, then the error estimate for the forward (respectively backward) finite difference scheme, if $p$ is odd (respectively even), yields 
\begin{eqnarray*}
\left|\left| e\right|\right|_{\ell^{\infty}\left(0,N; \ell^2_{\Delta}\left(\mathbb{Z}\right)\right)} \lesssim  \Delta t^{\frac{M}{4p+2}}||\partial_x^Mu_0||_{\mathbb{L}^2(\mathbb{R})}+\Delta x||\partial_x^{2p+2}u_0||_{\mathbb{L}^2(\mathbb{R})}.
\end{eqnarray*}
For the central difference scheme, we suppose $m\geq2p+3$, and one has (for the same $M$)
\begin{eqnarray*}
\left|\left| e\right|\right|_{\ell^{\infty}\left(0,N; \ell^2_{\Delta}\left(\mathbb{Z}\right)\right)} \lesssim  \Delta t^{\frac{M}{4p+2}}||\partial_x^Mu_0||_{\mathbb{L}^2(\mathbb{R})}+\Delta x^2||\partial_x^{2p+3}u_0||_{\mathbb{L}^2(\mathbb{R})}.
\end{eqnarray*}
\end{Proposition}
Before proving this result, we introduce a regularization of $u_0$ thanks to mollifiers $\left(\varphi^{\delta}\right)_{\delta>0} $. Let $\chi$ be a $\mathcal{C}^{\infty}$--function such that
\begin{itemize}
\item $0\leq\chi\leq 1$,
\item $\chi\equiv 1$ in $[-\frac{1}{2},\frac{1}{2}]$ and $\mathrm{supp}(\chi)\subset[-1,1]$ (where $\mathrm{supp}$ is its support),
\item $\chi(-\xi)=\chi(\xi)$, $\forall \xi\in[-1,1]$.
\end{itemize}
Let $\varphi$ be such as $\widehat{\varphi}\left(\xi\right)=\chi\left(\xi\right)$ and for all $\delta>0$, we define $\varphi^{\delta}$ such that $\widehat{\varphi^{\delta}}\left(\xi\right)=\chi\left(\delta\xi\right)$, which implies $\varphi^{\delta}=\frac{1}{\delta}\varphi\left(\frac{.}{\delta}\right)$.
 Eventually, \begin{itemize}
 \item let $u^{\delta}$ be the solution of (\ref{EQ_BASE}) with $u_0^{\delta}=u_0\star\varphi^{\delta}$ as initial data, where $\star$ stands for the convolution product. 
 \item We denote then $((v^{\delta})_j^n)_{(n,j)\in\{0,...,N\}\times\mathbb{Z}}$ the numerical solution obtained by applying the numerical scheme (\ref{EQ_NUMERIQUE}) from  $u_0^{\delta}$. 
 \item The unknowns $u$ and $(v_j^n)_{(n,j)\in\{0,...,N\}\times\mathbb{Z}}$ are always the exact and numerical solutions starting from the initial data $u_0$.
 \end{itemize}
 \begin{Lemma}
Assume $u_0 \in \mathbb{H}^{r}(\mathbb{R})$ with $r>0$ then the following upper bound holds, for $0\leq\ell\leq r\leq s$, \begin{small}\begin{eqnarray*}
\left|\left|u_0-u^{\delta}_0\right|\right|_{\mathbb{H}^{\ell}(\mathbb{R})}\lesssim \delta^{r-\ell}||u_0||_{\mathbb{H}^{r}(\mathbb{R})} \ \ \ and \ \ \ \left|\left|u_0^{\delta}\right|\right|_{\mathbb{H}^s(\mathbb{R})}\lesssim \frac{1}{\delta^{s-r}}||u_0||_{\mathbb{H}^{r}(\mathbb{R})}.
\end{eqnarray*}\end{small}
\label{lemme_2}
\end{Lemma}
\vspace*{-0.7cm}
\begin{proof}
Lemma~\ref{lemme_2} is proved in~\cite{Courtes_Lagoutiere_Rousset} and follows  from very classical arguments (see also \cite{Despres_2008}).
\end{proof}

\begin{proof}[Proof of Proposition~\ref{Prop_model}]
We are now able to prove Proposition~\ref{Prop_model}. The triangular inequality applied to the convergence error gives
$
||e^n||_{\ell^2_{\Delta}}\lesssim E_1+E_2+E_3
$
with\begin{small}
\begin{eqnarray}
E_1=\left(\sum_{j\in\mathbb{Z}}\Delta x \left(\frac{1}{\Delta x} \int_{x_j}^{x_{j+1}}u(t^n,x)-u^{\delta}(t^n,x)dx\right)^2\right)^{\frac{1}{2}},\label{DEF_E_1}
\end{eqnarray}
\vspace*{-0.6cm}
\begin{eqnarray}
E_2=\left(\sum_{j\in\mathbb{Z}}\Delta x \left(\frac{1}{\Delta x}\int_{x_j}^{x_{j+1}}u^{\delta}(t^n,x)dx-(v^{\delta})^n_j\right)^2\right)^{\frac{1}{2}},\label{DEF_E_2}
\end{eqnarray}
\vspace*{-0.6cm}
\begin{eqnarray}
E_3=\left(\sum_{j\in\mathbb{Z}}\Delta x\left((v^{\delta})^n_j-v^n_j\right)^2\right)^{\frac{1}{2}}.
\label{DEF_E_3}
\end{eqnarray}\end{small}
Cauchy--Schwarz inequality together with the conservation of the $\mathbb{L}^2$-norm (Remark~\ref{REM_1}) yield
$
E_1\leq ||u_0-u_0^{\delta}||_{\mathbb{L}^2(\mathbb{R})}\lesssim \delta^M ||\partial_x^Mu_0||_{\mathbb{L}^2(\mathbb{R})}.
$
The latest inequality comes from Lemma~\ref{lemme_2} with $\ell=0$ and $r=M$.

For the $E_2$--term, we use the previous section (Sect.~\ref{_Order_of_accuracy_for_a_smooth_initial_data_}). Indeed, $E_2$ corresponds to the convergence error for a smooth initial data $u_0^{\delta}$. Hence, one has
\begin{small}
\begin{eqnarray*}
E_2\lesssim\Delta t||\partial_x^{4p+2}u_0^{\delta}||_{\mathbb{L}^2(\mathbb{R})}+\Delta x||\partial_x^{2p+2}u_0^{\delta}||_{\mathbb{L}^2(\mathbb{R})}\lesssim \frac{\Delta t}{\delta^{4p+2-M}}||\partial_x^{M}u_0||_{\mathbb{L}^2(\mathbb{R})}+\Delta x||\partial_x^{2p+2}u_0||_{\mathbb{L}^{2}(\mathbb{R})},
\end{eqnarray*}
\end{small}
where the latest inequality comes from Lemma~\ref{lemme_2} with $(s,r)=(4p+2,M)$ and $(s,r)=(2p+2, 2p+2)$.

Finally, the stability of the scheme gives the following estimate for $E_3$ :
\begin{eqnarray*}
E_3=||(v^{\delta})^n-v^n||_{\ell_{\Delta}^2}\leq||u_0^{\delta}-u_0||_{\mathbb{L}^{2}(\mathbb{R})}.
\end{eqnarray*}
Thus, the convergence error is upper bounded by
\begin{eqnarray*}
\left|\left|e^n\right|\right|_{\ell^2_{\Delta}}\lesssim\delta^M||\partial_x^Mu_0||_{\mathbb{L}^2(\mathbb{R})}+\frac{\Delta t}{\delta^{4p+2-M}}||\partial_x^{M}u_0||_{\mathbb{L}^2(\mathbb{R})}+\Delta x||\partial_x^{2p+2}u_0||_{\mathbb{L}^{2}(\mathbb{R})}.
\end{eqnarray*}
Proposition~\ref{Prop_model} comes from the optimal choice for $\delta$ : $\delta=\Delta t^{\frac{1}{4p+2}}$.

\end{proof}
\begin{Remark}
The result for the central finite difference scheme is proved exactly in the same way, with the same $s,\ r,\ \ell$ and $\delta$.
\end{Remark}
\subsection{Initial datum in $\mathbb{H}^m(\mathbb{R})$ with $m\geq0$}
The main result of this paper is summarized in the following theorem where the initial data $u_0$ has any Sobolev regularity.
\begin{Theorem}
Assume that $u_0\in \mathbb{H}^m(\mathbb{R})$ with $0\leq m$, then, the forward (respectively backward) finite difference scheme if $p$ is odd (respectively even) has the following error-estimate
\begin{eqnarray*}
\left|\left| e\right|\right|_{\ell^{\infty}\left(0,N;\ell^2_{\Delta}\left(\mathbb{Z}\right)\right)} \lesssim \Delta t^{\frac{\mathrm{min}(m,4p+2)}{4p+2}}||\partial_x^{\mathrm{min}(4p+2,m)}u_0||_{\mathbb{L}^2(\mathbb{R})}+\Delta x^{\frac{\mathrm{min}(m,2p+2)}{2p+2}} ||\partial_x^{\mathrm{min}(2p+2,m)}u_0||_{\mathbb{L}^2(\mathbb{R})}.
\end{eqnarray*}
The previous inequality becomes for the central finite difference scheme
\begin{eqnarray*}
\left|\left| e\right|\right|_{\ell^{\infty}\left(0,N;\ell^2_{\Delta}\left(\mathbb{Z}\right)\right)} \lesssim \Delta t^{\frac{\mathrm{min}(m,4p+2)}{4p+2}} ||\partial_x^{\mathrm{min}(4p+2,m)}u_0||_{\mathbb{L}^2(\mathbb{R})}+\Delta x^{2\frac{\mathrm{min}(m,2p+3)}{2p+3}} ||\partial_x^{\mathrm{min}(2p+3,m)}u_0||_{\mathbb{L}^2(\mathbb{R})}.
\end{eqnarray*}
The previous results are summarized in Table~\ref{TABLE_ERROR_ESTIMATES}.
\label{16-prop}
\end{Theorem}
\begin{proof}

Here again, we suppose that $p$ is odd, we thus detail the proof for the forward finite difference scheme. We have already proved the case $m\geq 4p+2$ in Sect.~\ref{_Order_of_accuracy_for_a_smooth_initial_data_} and the case $2p+2\leq m\leq 4p+2$ in Subsect.~\ref{SUB_1_1}. Let us now focus on the case $0\leq m\leq 2p+2$.

The proof of Theorem~\ref{16-prop} follows the same guidelines as the proof of Proposition~\ref{Prop_model}. Let $u_0\in\mathbb{H}^m(\mathbb{R})$, we regularize this initial data thanks to mollifiers $\left(\varphi^{\delta}\right)_{\delta>0}$ whose properties are listed in Subsect.~\ref{SUB_1_1}. This involves introducing the same new unknowns $u^{\delta}$, $u_0^{\delta}$ and $((v^{\delta})_j^n)_{(j,n)}$.

The convergence error $\left(e_j^n\right)_{(j,n)}$ is upper bounded by the same $E_1$, $E_2$ and $E_3$, defined in (\ref{DEF_E_1})-(\ref{DEF_E_3}). Lemma~\ref{lemme_2} with $\ell=0$ and $r=m$ leads to
$
E_1+E_3\lesssim \delta^{m}||\partial_x^m u_0||_{\mathbb{L}^2(\mathbb{R})}.
$
By definition $u_0^{\delta}\in\mathbb{H}^{k}(\mathbb{R})$, $\forall k>0$, therefore Proposition~\ref{Prop_model} applies with $k=2p+2$ for example and $M=\mathrm{min}(k,2p+2)=2p+2$. It gives the following estimate for $E_2$ :
\begin{small}
$
E_2\lesssim \Delta t^{\frac{p+1}{2p+1}}||\partial_x^{2p+2}u_0^{\delta}||_{\mathbb{L}^{2}(\mathbb{R})}+\Delta x||\partial_x^{2p+2}u_0^{\delta}||_{\mathbb{L}^2(\mathbb{R})}\lesssim \left(\Delta t^{\frac{p+1}{2p+1}}+\Delta x\right)||\partial_x^{2p+2}u_0^{\delta}||_{\mathbb{L}^{2}(\mathbb{R})}.
$
\end{small}
We then apply Lemma~\ref{lemme_2} with $s=2p+2$ and $r=m$. Finally, it yields
\begin{eqnarray*}
\left|\left|e^n\right|\right|_{\ell^2_{\Delta}}\lesssim \delta^{m}||\partial_x^m u_0||_{\mathbb{L}^2(\mathbb{R})}+\frac{\left(\Delta t^{\frac{p+1}{2p+1}}+\Delta x\right)}{\delta^{2p+2-m}}||\partial_x^{m}u||_{\mathbb{L}^2(\mathbb{R})}.
\end{eqnarray*}
The conclusion comes from the optimal choice $\delta=\left(\Delta t^{\frac{p+1}{2p+1}}+\Delta x\right)^{\frac{1}{2p+2}}$.

\end{proof}
\begin{Remark}
The backward scheme, with $p$ even, is very similar. The central finite difference scheme is proved with the same method except for the variable $k$, which is taken $k=2p+3$ for that scheme.
\end{Remark}
\section{Numerical results}
\label{_Numerical_results_}
In order to illustrate numerically the previous results, we perform two sets of examples : on the one hand, we compute the numerical rate of convergence of various equations for a fix initial data and one the other hand, the equation is fixed and we test different initial data.\\
In all examples, the computational domain is set to $[0,50]$ subdivided into $J$ cells with $$J\in\{800, 1600, 3200, 6400, 12800, 25600, 51200, 102400\}$$ and the numerical simulation is performed up to time $T=0.1$. Not to have a too restricted Courant-Friedrich-Lewy condition, we implement the implicit scheme ($\theta=1$) and impose  $\Delta t=\Delta x$. The convergence error is computed between the solution with $J$ cells and a 'reference' solution with $2J$ cells in space.\\
Since the indicator function belongs to $\mathbb{H}^s(\mathbb{R})$ for all $s<\frac{1}{2}$, we build test functions in $\mathbb{H}^{s}(\mathbb{R})$ with $ s<\frac{1}{2}+k$ by integrating $k$-times the indicator function. Such functions will be denoted in $\mathbb{H}^{\frac{1}{2}+k-}(\mathbb{R})$.\\
The first test consists of fixing $u_0$ in $\mathbb{H}^{\frac{3}{2}-}$ or $\mathbb{H}^{\frac{5}{2}-}$ and compute the convergence rate for $p=0$ (advection equation), $p=1$ (Airy equation) and $p=2$. The numerical results are gathered in Tables~\ref{H3demi} and \ref{H5demi} and correctly match with the expected theoretical rates.
\begin{table}
\begin{center}
%
%
\begin{tabular}{p{1.6cm}p{1.6cm}p{1.6cm}p{1.6cm}p{1.6cm}p{1.6cm}p{1.6cm}}
\hline\noalign{\smallskip}
&\multicolumn{2}{c}{\hspace*{-1cm}$p=0$ }&\multicolumn{2}{c}{\hspace*{-1cm}$p=1$}&\multicolumn{2}{c}{\hspace*{-1cm}$p=2$}\\
$\Delta x$&\multicolumn{2}{c}{\hspace*{-1cm}(Advection)}&\multicolumn{2}{c}{\hspace*{-1cm}(Airy)}&\multicolumn{2}{c}{\hspace*{-1cm}(Fifth-order derivative)}\\
\hline\noalign{\smallskip}
&$\mathbb{L}^2$-error&order&$\mathbb{L}^2$-error&order&$\mathbb{L}^2$-error&order\\
$6.250.10^{-2}$&&&&&$2.985.10^{-3}$&\\
$3.125.10^{-2}$&&&&&$2.757.10^{-3}$&$0.115$\\
$1.563.10^{-2}$&&&&&$2.441.10^{-3}$&$0.175$\\
$7.813.10^{-3}$&$1.194.10^{-4}$&&$1.348.10^{-3}$&&$2.176.10^{-3}$&$0.166$\\
$3.906.10^{-3}$&$7.381.10^{-5}$&$0.694$&$1.125.10^{-3}$&$0.261$&$1.961.10^{-3}$&$0.149$\\
$1.953.10^{-3}$&$4.471.10^{-5}$&$0.723$&$9.670.10^{-4}$&$0.219$&&\\
$9.766.10^{-4}$&$2.664.10^{-5}$&$0.747$&$7.968.10^{-4}$&$0.279$&&\\
$4.883.10^{-4}$&$1.585.10^{-5}$&$0.749$&$6.572.10^{-4}$&$0.278$&&\\
\hline\noalign{\smallskip}
theoretical&&$0.750$&&$0.250$&&$0.150$\\
\hline\noalign{\smallskip}
\end{tabular}
\caption{For a Sobolev regularity $\mathbb{H}^{\frac{3}{2}-}$}
\label{H3demi}       
\end{center}
\end{table}
\begin{table}
\begin{center}
%
%
\begin{tabular}{p{1.6cm}p{1.6cm}p{1.6cm}p{1.6cm}p{1.6cm}p{1.6cm}p{1.6cm}}
\hline\noalign{\smallskip}
&\multicolumn{2}{c}{\hspace*{-1cm}$p=0$ }&\multicolumn{2}{c}{\hspace*{-1cm}$p=1$}&\multicolumn{2}{c}{\hspace*{-1cm}$p=2$}\\
$\Delta x$&\multicolumn{2}{c}{\hspace*{-1cm}(Advection)}&\multicolumn{2}{c}{\hspace*{-1cm}(Airy)}&\multicolumn{2}{c}{\hspace*{-1cm}(Fifth-order derivative)}\\
\hline\noalign{\smallskip}
&$\mathbb{L}^2$-error&order&$\mathbb{L}^2$-error&order&$\mathbb{L}^2$-error&order\\
$6.250.10^{-2}$&&&&&$3.388.10^{-2}$&\\
$3.125.10^{-2}$&&&&&$3.639.10^{-2}$&$0.093$\\
$1.563.10^{-2}$&&&&&$3.032.10^{-2}$&$0.263$\\
$7.813.10^{-3}$&$2.586.10^{-3}$&&$1.011.10^{-2}$&&$2.528.10^{-2}$&$0.262$\\
$3.906.10^{-3}$&$1.347.10^{-3}$&$0.940$&$7.507.10^{-3}$&$0.430$&$2.138.10^{-2}$&$0.242$\\
$1.953.10^{-3}$&$6.873.10^{-4}$&$0.971$&$6.267.10^{-3}$&$0.261$&&\\
$9.766.10^{-4}$&$3.437.10^{-4}$&$0.999$&$4.401.10^{-3}$&$0.510$&&\\
$4.883.10^{-4}$&$1.719.10^{-4}$&$1.000$&$3.074.10^{-3}$&$0.520$&&\\
\hline\noalign{\smallskip}
theoretical&&$1.000$&&$0.417$&&$0.250$\\
\hline\noalign{\smallskip}
\end{tabular}
\caption{For a Sobolev regularity $\mathbb{H}^{\frac{5}{2}-}$}
\label{H5demi}       
\end{center}
\end{table}
For the second sample of examples, the equation is fixed ($p=0$ for Fig.~\ref{fig_synthese}-left and $p=1$ for Fig.~\ref{fig_synthese}-right) whereas the Sobolev regularity of the initial data is fluctuating. As shown in Fig.~\ref{fig_synthese}, the theoretical rates are represented by the line and the numerical rates correspond to the dot. The exponent of the Sobolev regularity of $u_0$ is shown in the x-axis. Again, the different rates match very well, which tends to indicate that the convergence orders we have proven are optimal.
\begin{figure}
\begin{center}
\includegraphics[scale=.4]{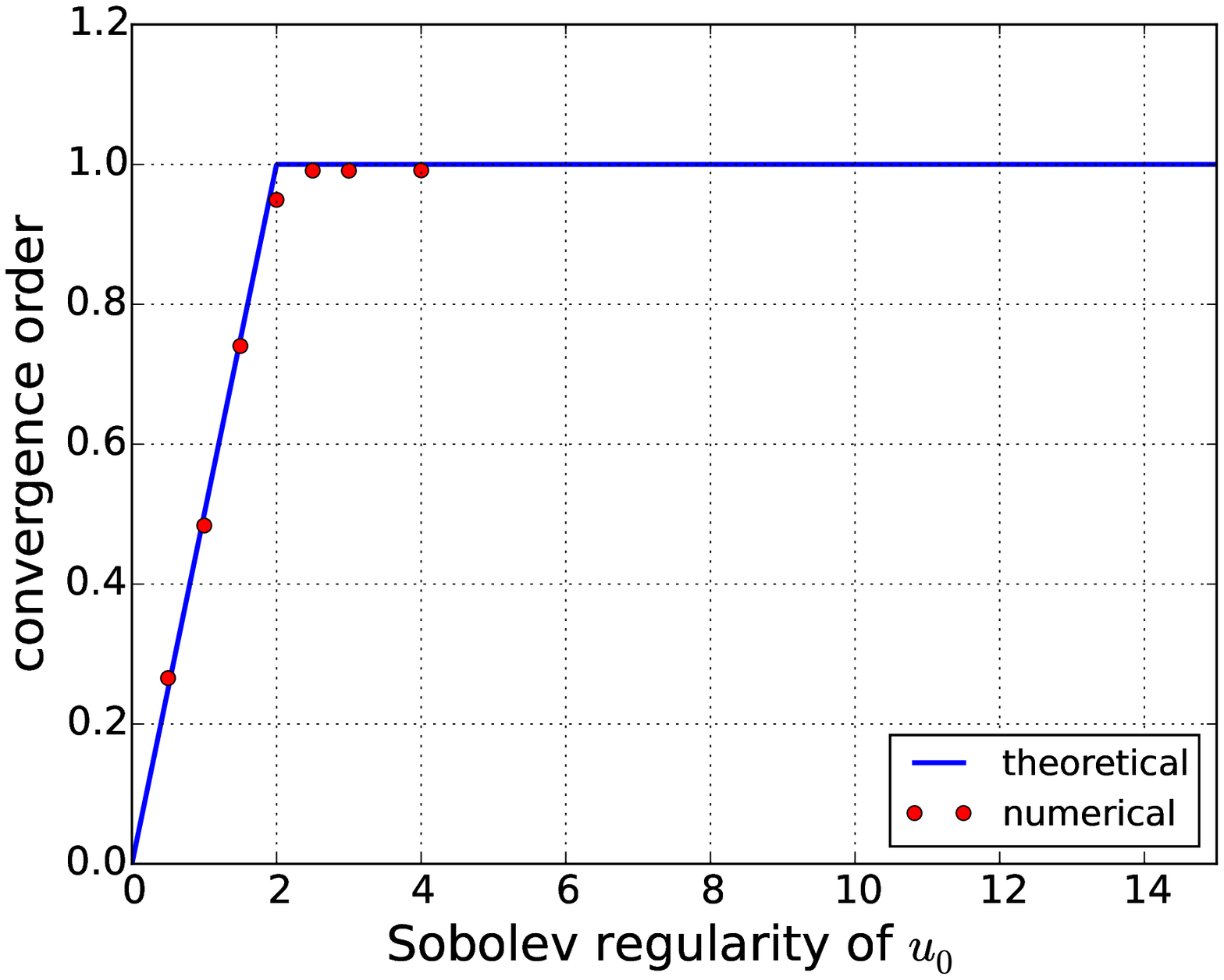}
\includegraphics[scale=.4]{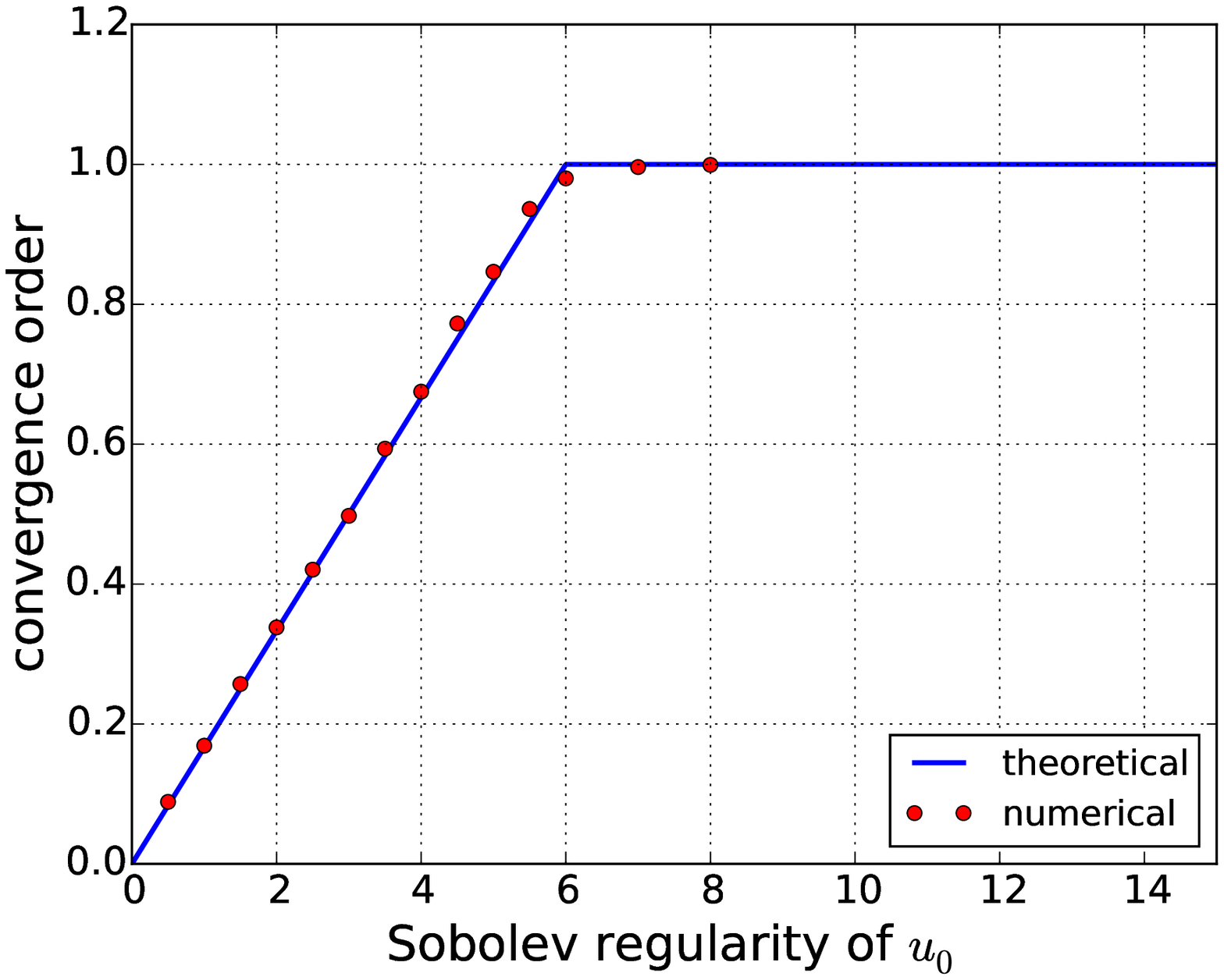}
\caption{Numerical versus theoretical orders--- left : Advection equation ($p=0$), right : Airy equation ($p=1$)}
\label{fig_synthese}
\end{center}
\end{figure}

\begin{landscape}
\begin{table}
\begin{tabular}{p{2cm}p{5cm}p{5cm}p{5cm}}
\hline\noalign{\smallskip}
For $\theta\neq \frac{1}{2}$& \multicolumn{2}{c}{\hspace*{-2cm}$\overbrace{\hspace*{8cm}}^{p \mathrm{\ even}}$} &$p$ odd\\
& $p=0$ (Advection)& $p\neq 0$&  \\
\hline\noalign{\smallskip}
Forward & $\Delta t^{\frac{\mathrm{min}(m,2)}{2}}\left|\left| \partial_x^{\mathrm{min}(m,2)}u_0\right|\right|_{\mathbb{L}^2}$ & &$\Delta t^{\frac{\mathrm{min}(m,4p+2)}{4p+2}}\left|\left| \partial_x^{\mathrm{min}(m,4p+2)}u_0\right|\right|_{\mathbb{L}^2}$\\
scheme &$+\Delta x^{\frac{\mathrm{min}(m,2)}{2}} \left|\left| \partial_x^{\mathrm{min}(m,2)}u_0\right|\right|_{\mathbb{L}^2}$&&$+\Delta x^{\frac{\mathrm{min}(m,2p+2)}{2p+2}} \left|\left| \partial_x^{\mathrm{min}(m,2p+2)}u_0\right|\right|_{\mathbb{L}^2}$\\
&&&\\
Backward &$\Delta t^{\frac{\mathrm{min}(m,2)}{2}}\left|\left| \partial_x^{\mathrm{min}(m,2)}u_0\right|\right|_{\mathbb{L}^2}$  & $\Delta t^{\frac{\mathrm{min}(m,4p+2)}{4p+2}}\left|\left| \partial_x^{\mathrm{min}(m,4p+2)}u_0\right|\right|_{\mathbb{L}^2}$&\\
scheme &$+\Delta x^{\frac{\mathrm{min}(m,2)}{2}} \left|\left| \partial_x^{\mathrm{min}(m,2)}u_0\right|\right|_{\mathbb{L}^2}$&$+\Delta x^{\frac{\mathrm{min}(m,2p+2)}{2p+2}} \left|\left| \partial_x^{\mathrm{min}(m,2p+2)}u_0\right|\right|_{\mathbb{L}^2}$&\\
&&&\\
Central & $\Delta t^{\frac{\mathrm{min}(m,2)}{2}}\left|\left|\partial_x^{\mathrm{min}(m,2)}u_0\right|\right|_{\mathbb{L}^2}$ & $\Delta t^{\frac{\mathrm{min}(m,4p+2)}{4p+2}}\left|\left| \partial_x^{\mathrm{min}(m,4p+2)}u_0\right|\right|_{\mathbb{L}^2}$&$\Delta t^{\frac{\mathrm{min}(m,4p+2)}{4p+2}}\left|\left| \partial_x^{\mathrm{min}(m,4p+2)}u_0\right|\right|_{\mathbb{L}^2}$\\
scheme &$+\Delta x^{2\frac{\mathrm{min}(m,3)}{3}}\left|\left|\partial_x^{\mathrm{min}(m,3)}u_0\right|\right|_{\mathbb{L}^2}$&$+\Delta x^{2\frac{\mathrm{min}(m,2p+3)}{2p+3}}\left|\left|\partial_x^{\mathrm{min}(m,2p+3)}u_0\right|\right|_{\mathbb{L}^2}$&$+\Delta x^{2\frac{\mathrm{min}(m,2p+3)}{2p+3}}\left|\left|\partial_x^{\mathrm{min}(m,2p+3)}u_0\right|\right|_{\mathbb{L}^2}$\\
\hline\noalign{\smallskip}
\end{tabular}

\begin{tabular}{p{2cm}p{5cm}p{5cm}p{5cm}}
\hline\noalign{\smallskip}
For $\theta=\frac{1}{2}$& (Crank-Nicolson case)&&\\
\hline\noalign{\smallskip}
Forward & $\Delta t^{2\frac{\mathrm{min}(m,3)}{3}}\left|\left| \partial_x^{\mathrm{min}(m,3)}u_0\right|\right|_{\mathbb{L}^2(\mathbb{R})}$ & &$\Delta t^{2\frac{\mathrm{min}(m,6p+3)}{6p+3}}\left|\left| \partial_x^{\mathrm{min}(m,6p+3)}u_0\right|\right|_{\mathbb{L}^2(\mathbb{R})}$\\
scheme &$+\Delta x^{\frac{\mathrm{min}(m,2)}{2}} \left|\left| \partial_x^{\mathrm{min}(m,2)}u_0\right|\right|_{\mathbb{L}^2(\mathbb{R})}$&&$+\Delta x^{\frac{\mathrm{min}(m,2p+2)}{2p+2}} \left|\left| \partial_x^{\mathrm{min}(m,2p+2)}u_0\right|\right|_{\mathbb{L}^2(\mathbb{R})}$\\
&&&\\
Backward &$\Delta t^{2\frac{\mathrm{min}(m,3)}{3}}\left|\left| \partial_x^{\mathrm{min}(m,3)}u_0\right|\right|_{\mathbb{L}^2(\mathbb{R})}$  & $\Delta t^{2\frac{\mathrm{min}(m,6p+3)}{6p+3}}\left|\left| \partial_x^{\mathrm{min}(m,6p+3)}u_0\right|\right|_{\mathbb{L}^2(\mathbb{R})}$&\\
scheme &$+\Delta x^{\frac{\mathrm{min}(m,2)}{2}} \left|\left| \partial_x^{\mathrm{min}(m,2)}u_0\right|\right|_{\mathbb{L}^2(\mathbb{R})}$&$+\Delta x^{\frac{\mathrm{min}(m,2p+2)}{2p+2}} \left|\left| \partial_x^{\mathrm{min}(m,2p+2)}u_0\right|\right|_{\mathbb{L}^2(\mathbb{R})}$&\\
&&&\\
Central & $\Delta t^{2\frac{\mathrm{min}(m,3)}{3}} \left|\left|\partial_x^{\mathrm{min}(m,3)}u_0\right|\right|_{\mathbb{L}^2(\mathbb{R})}$ & $\Delta t^{2\frac{\mathrm{min}(m,6p+3)}{6p+3}}\left|\left| \partial_x^{\mathrm{min}(m,6p+3)}u_0\right|\right|_{\mathbb{L}^2(\mathbb{R})}$&$\Delta t^{2\frac{\mathrm{min}(m,6p+3)}{6p+3}}\left|\left| \partial_x^{\mathrm{min}(m,6p+3)}u_0\right|\right|_{\mathbb{L}^2(\mathbb{R})}$\\
scheme &$+\Delta x^{2\frac{\mathrm{min}(m,3)}{3}}\left|\left|\partial_x^{\mathrm{min}(m,3)}u_0\right|\right|_{\mathbb{L}^2(\mathbb{R})}$&$+\Delta x^{2\frac{\mathrm{min}(m,2p+3)}{2p+3}}\left|\left|\partial_x^{\mathrm{min}(m,2p+3)}u_0\right|\right|_{\mathbb{L}^2(\mathbb{R})}$&$+\Delta x^{2\frac{\mathrm{min}(m,2p+3)}{2p+3}}\left|\left|\partial_x^{\mathrm{min}(m,2p+3)}u_0\right|\right|_{\mathbb{L}^2(\mathbb{R})}$\\
\hline\noalign{\smallskip}
\end{tabular}
\caption{Error estimates for $u_0\in\mathbb{H}^{m}(\mathbb{R})$}
\label{TABLE_ERROR_ESTIMATES}       
\end{table}
\end{landscape}

\end{document}